\documentclass{article}
\usepackage{amsmath,amsthm,amsopn,amssymb}
\usepackage{bbding}
\usepackage{graphicx}
\usepackage{hyperref}
\usepackage{mathptmx}
\allowdisplaybreaks

\numberwithin{equation}{section}
\newtheorem{theorem}{Theorem}[section]
\newtheorem{lemma}{Lemma}[section]
\makeatother

\begin{document}

\title{Global monotone convergence of Newton-like iteration for a nonlinear eigen-problem
\thanks{Research supported in part by National Natural Science Foundation of China (Grant No. 12001504 and Grant No. 41874165) and the Fundamental Research Funds for the Central Universities (Grant Number 2652019320)} }

\author{Pei-Chang Guo \thanks{ e-mail: guopeichang@pku.edu.cn;peichang@cugb.edu.cn}  \quad \\
School of Science, China University of Geosciences, Beijing, 100083, China\\
}

\date{}

\maketitle
\begin{abstract}
The nonlinear eigen-problem $ Ax+F(x)=\lambda x$ is studied where $A$ is an $n\times n$ irreducible  Stieltjes matrix.
Under certain conditions, this problem has a unique positive solution. We show that, starting from a multiple of the positive eigenvector of $A$, the Newton-like iteration
for this problem converges monotonically.
Numerical results illustrate the effectiveness of this Newton-like method.

\vspace{2mm} \noindent \textbf{Keywords}: eigen-problem, Stieltjes matrix, Newton-like method, monotone convergence
\end{abstract}

\section{Introduction}
We introduce some necessary notation for the paper. For any matrices $B=[b_{ij}]\in \mathbb{R}^{n\times n}$,
we write $B\geq 0 (B>0)$ if $b_{ij} \geq 0 (b_{ij}> 0)$ holds for all $i,j$.
For any matrices $A,B \in \mathbb{R}^{n\times n}$, we write $A\geq B (A > B)$ if $a_{ij}\geq b_{ij}(a_{ij} > b_{ij})$ for all $i, j$. For any vectors $x,y \in \mathbb{R}^n$ ,
we write $x\geq y (x>y)$ if $x_i\geq y_i (x_i>y_i)$ holds for all $i=1, \cdots,n$. The vector of all ones is denoted by $e$, i.e., $e=(1, 1, \cdots, 1)^T$.
The notation $\odot$ denotes the Hadamard product and $x^{[2]}$ denotes $x\odot x$.

In this paper, we studied the Newton-like  iteration for the following nonlinear eigenvalue problem:
\begin{equation}\label{eq0}
    Ax+F(x)=\lambda x,
\end{equation}
where
\begin{equation*}
    F(x)=\left(
           \begin{array}{c}
             f_1(x_1) \\
              f_2(x_2) \\
             \vdots \\
             f_n(x_n) \\
           \end{array}
         \right)
\end{equation*}
and $A$ is an $n\times n$ irreducible Stieltjes matrix, which is an irreducible symmetric positive definite matrix with all off-diagonal entries being nonpositive.
This problem \eqref{eq0} arises in the discretization of nonlinear differential equations. In the applications, this condition is satisfied: if $x=(x_1,x_2,\cdots, x_n)^{T}>0$ then
 $F(x)>0$ and
 $f^{''}(x)>0$ hold, where
 \begin{equation*}
    f^{''}(x)=\left(
           \begin{array}{c}
             f_1^{''}(x_1) \\
             f_2^{''}(x_2) \\
             \vdots \\
             f_n^{''}(x_n) \\
           \end{array}
         \right)
\end{equation*}

One example of \eqref{eq0} is the discretized  Gross-Pitaevskii  equation.
The Gross-Pitaevskii  equation is important when analyzing the Bose-Einstein condensation of atoms at near absolute zero temperatures \cite{bose1999}.
The continuous  Gross-Pitaevskii  equation in three spatial variables has the form
\begin{eqnarray*}
  \nonumber &-\triangle u +v(r,s,t)+ku^3=\lambda u,\quad k>0\\
  \nonumber &\lim_{|(r,s,t)|\rightarrow\infty} u=0,\quad \int_{-\infty}^{\infty}\int_{-\infty}^{\infty}\int_{-\infty}^{\infty}u(r,s,t)^2 dr ds dt=1
\end{eqnarray*}
where $v$ is a given potential function.

It's shown in \cite{choi2001} that under certain constraints on $F(x)$, \eqref{eq0} has a positive solution, $x(\lambda)$, if and only if $\lambda>\mu$, where $\mu$ is the smallest eigenvalue of $A$.
And $x(\lambda)$ is unique and monotonically increasing as $\lambda$ increases. In \cite{choi2002}, the Newton method is developed for the eigen-problem \eqref{eq0}.
At every Newton iteration step, we need to compute the Fr\'{e}chet derivative and solve a linear system. See more about the Newton method for other matrix equations in \cite{guo2013,jia2015,kelley,lin2008,liu2020}.
In this paper, motivated by the results in \cite{guo2013,jia2015,lin2008,liu2020}, we will present the Newton-like method for \eqref{eq0} and prove the global monotone convergence theory.

The rest of the paper is organized as follows. In section 2 we recall Newton's method and present a modified Newton iterative procedure. In section 3 we prove the monotone convergence
 for the modified Newton method. In section 4 we present some
numerical results, which show that our new algorithm can be faster
than the Newton method. In section 5, we give our conclusions.

\section{Newton-like  methods}
We write the eigenvalue problem as
\begin{equation}\label{eigeq}
    \mathcal{G}(x)=Ax+F(x)-\lambda x=0
\end{equation}
The function $\mathcal{G}$ defined in \eqref{eigeq} is a mapping from $\mathbb{R}^n$ into
itself and the Fr\'{e}chet derivative of $\mathcal{G}$ at $x$ is a linear map $\mathcal{G}^{'}_x: \mathbb{R}^n\rightarrow\mathbb{R}^n$ given by
\begin{equation}\label{daoshu}
   \mathcal{G}^{'}_x: z\mapsto Az+f^{'}(x)\odot z-\lambda z=(A+D^{(x)}-\lambda I)z,
\end{equation}
where
 \begin{equation*}
    f^{'}(x)=\left(
           \begin{array}{c}
             f_1^{'}(x_1) \\
             f_2^{'}(x_2) \\
             \vdots \\
             f_n^{'}(x_n) \\
           \end{array}
         \right)
\end{equation*}
and $D^{(x)}$ is a  square diagonal matrix with the elements of vector $f^{'}(x)$ on the main diagonal, i.e.,
\begin{equation*}
D^{(x)}=\left(
  \begin{array}{cccc}
    f_1^{'}(x_1) &0 & \cdots & 0 \\
    0 & f_2^{'}(x_2) &\ddots & \vdots \\
   \vdots & \ddots & \ddots & 0 \\
    0 & \cdots & 0 & f_n^{'}(x_n) \\
  \end{array}
\right).
\end{equation*}
To suppress the technical details, later we will consider $ \mathcal{G}^{'}_x$ and the matrix $ A+D^{(x)}-\lambda I$ as equal.
The second derivative of $\mathcal{G}$ at $x$, $\mathcal{G}^{''}_x: \mathbb{R}^n \times \mathbb{R}^n \rightarrow\mathbb{R}^n$, is given by
\begin{equation}\label{erdaoshu}
   \mathcal{G}^{''}_x(y,z)=f^{''}(x)\odot y\odot z.
\end{equation}
We present a  modified Newton method for \eqref{eigeq} as follows.

\textbf{ The Modified Newton Method for Equation \eqref{eq0}}\\

For a fixed $\lambda$, given $x_{0,0}$, for $k=0,1,2,\cdots$
    \begin{eqnarray}
        \label{eiglikea} x_{k,s} &=& x_{k,s-1}-(\mathcal{G}^{'}_{x_{k,0}})^{-1}\mathcal{G}(x_{k,s-1}), \quad s=1, \ldots, n_k, \\
        \label{eiglikeb} x_{k+1,0} &=& x_{k,n_k}.
\end{eqnarray}
When $n_k=1$ holds for $k=0,1,2,\cdots$, the modified Newton method \eqref{eiglikea},\eqref{eiglikeb} is the standard Newton iteration.
\section{Convergence Analysis}
We first recall that a real square matrix $A$ is called a Z-matrix if
all its off-diagonal elements are nonpositive. Note that any Z-matrix A can be
written as $sI-B$ with $B \geq 0$. A Z-matrix $A$ is called an M-matrix if $s\geq \rho(B)$,
where $\rho(\cdot)$ is the spectral radius; it is a singular M-matrix if $s=\rho(B)$ and a
nonsingular M-matrix if $s>\rho(B)$.
We will make use of the following result (see \cite{varga}).
\begin{lemma}\label{lemma1}
For a Z-matrix $A$, the following are equivalent:
\begin{itemize}
  \item [$(a)$] $A$ is a nonsingular M-matrix.
  \item [$(b)$] $A^{-1}\geq 0$ .
  \item [$(c)$] $Av>0$ for some vector $v>0$.
  \item [$(d)$] All eigenvalues of $A$ have positive real parts.
\end{itemize}
\end{lemma}
The following statement can be found in \cite{varga}.
\begin{lemma}\label{lemma01}
 If $A$ is a Stieltjes matrix, then it's also an M-matrix. Moreover, A is irreducible if and only if $A^{-1}>0$.
 \end{lemma}
The next result is also well known and also can be found in \cite{varga}.
\begin{lemma}\label{lemma2}
Let $A$ be a nonsingular M-matrix. If $B \geq A$ is a Z-matrix, then $B$ is also  nonsingular M-matrix. Moreover, $B^{-1}\leq A^{-1}$.
\end{lemma}

The following results are given in \cite{choi2001,choi2002}.
\begin{lemma}\label{choi1}
Let $A$ be an irreducible Stieltjes matrix and let $p$ be the positive
eigenvector of $A$ corresponding to $\mu$, the smallest eigenvalue of $A$.  Let
\begin{equation*}
    F(x)=\left(
           \begin{array}{c}
             f_1(x_1) \\
             \vdots \\
             f_n(x_n) \\
           \end{array}
         \right)
\end{equation*}
where for $i=1, \cdots, n$, $f_i\in C^1[0,\infty)$, $\lim_{t\rightarrow\infty}f_i(t)/t=\infty$ and $\lim_{t\rightarrow0} f_i(t)/t=0$.
Then for $\lambda > \mu$, \eqref{eq0} has a positive solution. If, in addition, for  $i=1, \cdots, n$,
\begin{equation}\label{choicon1}
    \frac{f_i(s)}{s}<\frac{f_i(t)}{t} \quad if \quad  0<s<t,
\end{equation}
then the solution is unique.
\end{lemma}

\begin{lemma}\label{choi2}
In conditions of Lemma \ref{choi1}, if $f_i(t), i=1, \cdots, n$ satisfy the following condition
\begin{equation}\label{choicon2}
    f^{'}_i(t)>\frac{f_i(t)}{t} \quad t>0,
\end{equation}
then $x(\lambda)$ is differentiable as a function of $\lambda$. Moreover, let $q$ be the unique positive solution of \eqref{eq0} for a given $\lambda$.
\begin{equation*}
    Aq+F(q) =\lambda q
\end{equation*}
implies that $ A+D^{(q)}-\lambda I $ is an irreducible Stieltjes matrix, where $D^{(q)}= diag( f_1^{'}(q_1), f_2^{'}(q_2),\cdots,  f_n^{'}(q_n))$.
\end{lemma}
Obviously, if  $f_i^{''}(t)>0, 0<t<\infty$, then the condition \eqref{choicon1} and \eqref{choicon2} are all satisfied.
Lemma \ref{choi2} tells that $ \mathcal{G}^{'}_q$ is an irreducible Stieltjes matrix, which will be useful in our following derivations.

We first give the following lemma, which displays the monotone convergence properties of the Newton iteration.
 \begin{lemma}\label{lemma3}
For $ i=1, \cdots, n$, assume $f_i^{''}(t)>0, 0<t<\infty$. Suppose that a vector $x$ is such that
\begin{itemize}
  \item [(i)] $\mathcal{G}(x)> 0$,
  \item [(ii)] $q<x $,
  \item [(iii)] $\mathcal{G}^{'}_{x}$ is an irreducible Stieltjes matrix.
\end{itemize}
Then the vector
\begin{equation}\label{zheng1}
    y=x-(\mathcal{G}^{'}_{x})^{-1}\mathcal{G}(x)
\end{equation}
satisfies
\begin{itemize}
  \item [(a)] $\mathcal{G}(y)> 0$,
  \item [(b)] $q<y<x$,
  \item [(c)]  $\mathcal{G}^{'}_{y}$ is an irreducible Stieltjes matrix.
\end{itemize}
\end{lemma}

\begin{proof}
First, the vector y is well defined by $(iii)$ and $y<x$. This is the right inequality of $(b)$.
By the Taylor formula, we get
\begin{equation*}
  0=\mathcal{G}(q)= \mathcal{G}(x)+ \mathcal{G}^{'}_{x}(q-x)+\int_{0}^{1}(1-\theta)f^{''}(x+\theta(q-x))\odot(q-x)^{[2]}d\theta
\end{equation*}
where $\theta \in \mathbb{R}$.
Therefore
\begin{eqnarray*}
\nonumber - \mathcal{G}^{'}_{x}(q-y)&=& \mathcal{G}^{'}_{x}(y-x)-\mathcal{G}^{'}_{x}(q-x) \\
\nonumber  &=& -\mathcal{G}(x)-\mathcal{G}^{'}_{x}(q-x) \\
\nonumber   &=&  \int_{0}^{1}(1-\theta)f^{''}(x+\theta(q-x))\odot(q-x)^{[2]}d\theta
\end{eqnarray*}
 where we used \eqref{zheng1} in the second equality and $\mathcal{G}(q)=0$ in the third equality.
When $0\leq \theta \leq1$, since $x>q>0$ by $(ii)$, we have $ x+\theta(q-x)=(1-\theta) x+\theta q>0$ and thus $f^{''}(x+\theta(q-x))>0$ by the assumption.  And we have $q<x $ by $(ii)$, so we can get
$\int_{0}^{1}(1-\theta)f^{''}(x+\theta(q-x))\odot(q-x)^{[2]}d\theta>0$, which is $ - \mathcal{G}^{'}_{x}(q-y)>0$.
$\mathcal{G}^{'}_{x}$ is an irreducible Stieltjes matrix by $(iii)$, so we have $(\mathcal{G}^{'}_{x})^{-1}>0$ and thus $q-y<0$. Up to now, we have proved $(b)$.

Because $\mathcal{G}^{'}_{q}$ is an irreducible Stieltjes matrix from Lemma \ref{choi2},  we have  $\mathcal{G}^{'}_{y}$ is an irreducible Stieltjes matrix by $y>q$ and Lemma \ref{lemma2}. This proves $(c)$.
Now we prove $(a)$.
\begin{eqnarray*}
  \nonumber \mathcal{G}(y)&=& \mathcal{G}(x)+ \mathcal{G}^{'}_{x}(y-x)+\int_{0}^{1}(1-\theta)f^{''}(x+\theta(y-x))\odot(y-x)^{[2]}d\theta\\
  \nonumber &=& \int_{0}^{1}(1-\theta)f^{''}(x+\theta(y-x))\odot(y-x)^{[2]}d\theta
\end{eqnarray*}
Since $(b)$ $0<q<y<x$ has been proved,  when $0\leq \theta \leq1$ we have $x+\theta(y-x)>0$  and thus $f^{''}(x+\theta(y-x))>0$. Then $\mathcal{G}(y)>0$.
\end{proof}

The following lemma is a generalization of Lemma \ref{lemma3} and will be the basis of our convergence result about Newton-like iteration.
 \begin{lemma}\label{lemma4}
For $ i=1, \cdots, n$, assume $f_i^{''}(t)>0, 0<t<\infty$. Suppose that a vector $x$ is such that
\begin{itemize}
  \item [(i)] $\mathcal{G}(x)> 0$,
  \item [(ii)] $q<x $,
  \item [(iii)] $\mathcal{G}^{'}_{x}$ is an irreducible Stieltjes matrix.
\end{itemize}
Then for any vector $z$ satisfying $z\geq x$,  the vector
\begin{equation}\label{zheng2}
    y=x-(\mathcal{G}^{'}_{z})^{-1}\mathcal{G}(x)
\end{equation}
is well defined and
\begin{itemize}
  \item [(a)] $\mathcal{G}(y)> 0$,
  \item [(b)] $q<y<x$,
  \item [(c)] $\mathcal{G}^{'}_{y}$ is an irreducible Stieltjes matrix.
\end{itemize}
\end{lemma}

\begin{proof}
Because $z\geq x$, $\mathcal{G}^{'}_{z}$ is also an irreducible Stieltjes matrix from Lemma \ref{lemma2} and $(iii)$.
So the vector $y$ is well defined and $y<x$ follows from $(\mathcal{G}^{'}_{z})^{-1}>0$ and $\mathcal{G}(x)>0$.
To prove $(b)$, next we need to prove $q<y$.
Let $\hat{y}=x-(\mathcal{G}^{'}_{x})^{-1}\mathcal{G}(x)$. $0< (\mathcal{G}^{'}_{z})^{-1}\leq (\mathcal{G}^{'}_{x})^{-1}$ holds from Lemma \ref{lemma2}, so
$y \geq \hat{y}$. By Lemma \ref{lemma3}, we have $\hat{y}>q$. So we get $q<\hat{y}\leq y<x$ then $(b)$ is true.

By   Lemma \ref{lemma3}, we also know $\mathcal{G}^{'}_{\hat{y}}$ is an irreducible Stieltjes matrix. So $(c)$ is true by $y\geq \hat{y}$ and Lemma \ref{lemma2}.
Next we prove $(a)$.
\begin{eqnarray*}
  \nonumber \mathcal{G}(y)&=& \mathcal{G}(x)+ \mathcal{G}^{'}_{x}(y-x)+\int_{0}^{1}(1-\theta)f^{''}(x+\theta(y-x))\odot(y-x)^{[2]}d\theta\\
  \nonumber    &=& \mathcal{G}(x)+ \mathcal{G}^{'}_{z}(y-x)+(\mathcal{G}^{'}_{x}-\mathcal{G}^{'}_{z})(y-x)+\int_{0}^{1}(1-\theta)f^{''}(x+\theta(y-x))\odot(y-x)^{[2]}d\theta\\
  \nonumber &=& (\mathcal{G}^{'}_{x}-\mathcal{G}^{'}_{z})(y-x)+\int_{0}^{1}(1-\theta)f^{''}(x+\theta(y-x))\odot(y-x)^{[2]}d\theta\\
  \nonumber &=& (f^{'}_{x}-f^{'}_{z})\odot (y-x)+\int_{0}^{1}(1-\theta)f^{''}(x+\theta(y-x))\odot(y-x)^{[2]}d\theta\\
   \nonumber &=& (f^{''}(z+w\odot(x-z))\odot(x-z)\odot (y-x)+\int_{0}^{1}(1-\theta)f^{''}(x+\theta(y-x))\odot(y-x)^{[2]}d\theta
\end{eqnarray*}
where $0<w<e, w\in\mathbb{R}^{n}$.
Because $x-z\leq0$, $y-x<0$ and $f^{''}(z+w\odot (x-z))>0$, $f^{''}(x+\theta(y-x))>0$, we have $\mathcal{G}(y)>0$.
\end{proof}

Using Lemma \ref{lemma4}, we can get the monotone convergence result of the Newton-like method for \eqref{eq0}. For  $i=0,1,\cdots$, we will use $x_i$ to
denote $x_{i,0}$ in the Newton-like method \eqref{eiglikea}\eqref{eiglikeb}, thus $x_i=x_{i,0}= x_{i-1,n_{i-1}}$.
\begin{theorem}\label{thm1}
Let $A$ be an irreducible Stieltjes matrix and
\begin{equation*}
    F(x)=\left(
           \begin{array}{c}
             f_1(x_1) \\
             \vdots \\
             f_n(x_n) \\
           \end{array}
         \right)
\end{equation*}
where for $i=1, \cdots, n$, $f_i\in C^2[0,\infty)$, $\lim_{t\rightarrow\infty}f_i(t)/t=\infty, \lim_{t\rightarrow0} f_i(t)/t=0$ and $f_i^{''}(t)>0, 0<t<\infty$. Let $p$ be the positive
eigenvector of $A$ corresponding to $\mu$, the smallest eigenvalue of $A$. For some $\lambda>\mu$, let $q$ be the unique positive solution of
\begin{equation*}
   \mathcal{G}(x)= Ax+F(x)-\lambda x=0.
\end{equation*}
Let $x_{0,0}=\beta p>q$, where $\beta$ is large enough  satisfying that
\begin{equation}\label{con1}
    \min_{1\leq i\leq n}\frac{f_i(\beta p_i)}{\beta p_i}>\lambda-\mu.
\end{equation}
Then the Newton-like method \eqref{eiglikea},\eqref{eiglikeb} converges to $q$ monotonically,
$x_{k+1} < x_{k}$ for all $k \geq 0$, and $ \lim_{k \to \infty} \mathcal{G}(x_k)=0$.

\end{theorem}
\begin{proof}
We prove the theorem by mathematical induction. From \eqref{con1}, we have $\mathcal{G}(x_{0,0})>0$. From $\lim_{t\rightarrow0} f_i(t)/t=0$ and $f_i^{''}(t)>0, 0<t<\infty$,
we have $f^{'}(x_{0,0})>0$, and thus $\mathcal{G}^{'}_{x_{0,0}}$ is an irreducible Stieltjes matrix.

So according to Lemma \ref{lemma4}, we have
\begin{equation*}
  q< x_1 = x_{0,n_0}<x_{0,n_0-1} \cdots < x_{0,1}< x_{0,0}
\end{equation*}
\begin{equation*}
    \mathcal{G}(x_1)> 0,
\end{equation*}
and
$\mathcal{G}^{'}_{x_1}$ is an irreducible Stieltjes matrix.

Assume
\begin{equation*}
  q< x_i=x_{i-1,n_{i-1}} \cdots <x_{i-1,1}<x_{i-1}
\end{equation*}
\begin{equation*}
    \mathcal{G}(x_i)> 0,
\end{equation*}
and
$\mathcal{G}^{'}_{x_i}$ is an irreducible Stieltjes matrix.
Again by Lemma \ref{lemma4} we have
\begin{equation*}
  q< x_{i+1} = x_{i,n_i}<x_{i,n_i-1} \cdots < x_{i,1}< x_{i}
\end{equation*}
\begin{equation*}
    \mathcal{G}(x_{i+1})> 0,
\end{equation*}
and
$\mathcal{G}^{'}_{x_{i+1}}$ is an irreducible Stieltjes matrix.

Therefore we have proved inductively the sequence $\{x_k\}$ is monotonically decreasing and bounded below. So it has a limit $x_*$. Next we show that $\mathcal{G}(x_*)=0$.
Since $q<x_k< x_0$, we have $\mathcal{G}^{'}_{x_k}\leq \mathcal{G}^{'}_{x_0}$ and we know $\mathcal{G}^{'}_{x_k}$ and $\mathcal{G}^{'}_{x_0}$  are irreducible Stieltjes matrices
by Lemma \ref{lemma2} and the fact that $\mathcal{G}^{'}_{q}$ is an irreducible Stieltjes matrix.  Then from Lemma \ref{lemma2} we have

$$(\mathcal{G}^{'}_{x_k})^{-1} \geq (\mathcal{G}^{'}_{x_0})^{-1} > 0.$$

Letting $i\rightarrow \infty$ in $x_{i+1}\leq x_{i,1}=x_i-(\mathcal{G}^{'}_{x_i})^{-1}\mathcal{G}({x_i})\leq x_i-(\mathcal{G}^{'}_{x_0})^{-1}\mathcal{G}({x_i})$,
we get $$\lim_{i\rightarrow \infty}(\mathcal{G}^{'}_{x_0})^{-1}\mathcal{G}({x_i})=0.$$
$\mathcal{G}(x)$ is continuous at $x_*$, so $(\mathcal{G}^{'}_{x_0})^{-1}\mathcal{G}({x_*})=0,$ and thus we get
$\mathcal{G}({x_*})=0.$

\end{proof}

\section{Numerical Experiments}
We remark that the modified Newton method differs from Newton's method in that the evaluation  of the  Fr\'{e}chet derivative are not done at every iteration step.
So, while more iterations will be needed than Newton's method, the overall cost of the modified Newton method may be much less.
Our numerical experiments confirm the efficiency of the modified Newton method for equation \eqref{eigeq}.
About how to choose the optimal scalars $n_i$ in the Newton-like  method(\ref{eiglikea}), we have no theoretical results for the moment.
This is a goal for our future research.

 The numerical  tests were performed on a laptop (3.0 Ghz and 16G Memory) with MATLAB R2016b. Numerical experiments show that the the modified Newton method can be more efficient
  than the Newton iteration in \cite{choi2002}.
We define the number of the evaluation of the Fr\'{e}chet derivative in the algorithm as the outer iteration steps,
  which is $i+1$ when $s>0$ or $i$ when $s=0$ for an approximate solution $x_{i,s}$ in the modified Newton algorithm.
The outer iteration steps (denoted as ``outer-iter"),  the
elapsed CPU time in seconds (denoted as ``time"), and the normalized residual
(denoted as "NRes" ) are used to measure the effectiveness of our new method, where "NRes" is
defined as
\begin{equation*}
\mbox{NRes}=\frac{\parallel A\tilde{x}+F(\tilde{x})-\lambda \tilde{x} \parallel_{\infty}}{\parallel  A\tilde{x} \parallel_{\infty} + \parallel F(\tilde{x}) \parallel_{\infty}+\lambda \parallel \tilde{x}\parallel_{\infty} },
\end{equation*}
where $\parallel\cdot\parallel_{\infty}$ is the vector infinity-norm and $\tilde{x}$ is an approximate solution to \eqref{eigeq}.

In Table~\ref{tab1}, we present the numerical results for the one-dimensional case of the Gross-Pitaevskii equation, i.e.,
\begin{eqnarray}
  &-x^{''}(t)+V(t) x(t) +kx^3(t)  = \lambda x(t), \quad -\infty<t<\infty \\
  &x(\pm \infty)=0,\quad \int^{\infty}_{-\infty}x^2(t)dt=1
\end{eqnarray}
where we set $V(t)=t^2$, $k=1$ and $\lambda=2$. We truncate this equation on the interval $[-5,5]$ and use finite differences to discretize this it. Specifically,
choose $n$, let $h=10/(n+1)$, and let $t_i= -5+ih$, $i=0,1,\cdots, n+1$.
 We use $x_0=\beta p$ as the initial iteration value of the Newton-like method, where $\beta=10$ and $p$ is the positive eigenvector of $A$ of unit infinity norm.
 Now the conditions in Theorem \ref{thm1} are satisfied.
 The stopping criterion is
 \begin{equation*}
    \parallel A\tilde{x}+F(\tilde{x})-\lambda \tilde{x} \parallel_{\infty}<\epsilon
 \end{equation*}
 where $\epsilon=10^{-10}$.
In Table~\ref{tab1}, we update the Fr\'{e}chet derivative every three iteration steps.
That is, for $i=0,1,\cdots$  we choose $n_i=3$ in the Newton-like  method (\ref{eiglikea}).


\begin{table}[h]
\label{tab1}
\begin{center}
\caption{Comparison of the numerical results}\label{tab1}
\begin{tabular}{|c|c|c|c|c|}\hline
$n$ & Method &  time & NRes & outer-iter\\
\hline  & Newton&  0.1719 &  2.16e-13& 10 \\
\cline{2-5} 500 & modified Newton & 0.1250& 3.27e-12 & 6 \\
\hline  & Newton&  1.2031 &  9.34e-13& 10 \\
\cline{2-5} 1000 & modified Newton &1.0313 & 3.98e-12 & 6 \\

\hline
\end{tabular}
\end{center}
\end{table}

\section{Conclusions}
In this paper, we consider the modified Newton method for a nonlinear eigen-problem.
The convergence analysis shows that this modified Newton method is feasible.
 Numerical experiments show that the modified Newton method is effective and can outperform Newton's method.

\section*{Conflict of Interest}
 The authors declare that they have no conflict of interest.

\end{document}